\documentclass{amsart}

\usepackage{amssymb,amsfonts,amsmath,amsthm,amscd}

\newtheorem{thm}{Theorem}%[section]

\newtheorem{lemma}[thm]{Lemma}

\begin{document}

\title{Schur indices in GAP: {\tt wedderga 4.6+} }

\author{Allen Herman}
\address{Department of Mathematics and Statistics, University of Regina, Regina, Canada, S4S 0A2}
\email{aherman@math.uregina.ca}

\thanks{The author's research is supported by an NSERC Discovery Grant.  Computing facilities for this project were provided by the University of Regina's {\it Laboratory for Computational Discovery}: {\tt http://www.lcd.uregina.ca} }

%\keywords{Schur indices, Representation theory of finite groups}

%\subjclass[2000]{Primary 20C15, Secondary 11H34}

\begin{abstract} 
We describe a algorithms and their implementations that calculate local and global Schur indices of ordinary irreducible characters of finite groups, cyclotomic algebras over abelian number fields, and rational quaternion algebras.   These functions are available with the latest release of the GAP package {\tt wedderga}, versions 4.6 and higher.  
\end{abstract}

\maketitle

\section{Introduction}

The GAP package {\tt wedderga} \cite{W}, originally released in 2006, features an algorithm for obtaining a presentation of the Wedderburn decomposition of the group algebra of a finite group $G$ over $F$, where $F$ is either a finite field or an abelian number field.   Simple components of the Wedderburn decomposition obtained using {\tt wedderga} appear as cyclotomic crossed product algebras; i.e. crossed products over cyclotomic extensions of $F$ whose factor set is expressed entirely in terms of roots of unity.  The {\tt wedderga} package achieves this presentation of the Wedderburn decomposition of $FG$ using a search algorithm based on the Brauer-Witt theorem.   It searches for suitable subgroups of $G$ that can be used to generate the $p$-part of the simple component corresponding to a given irreducible character of $G$ for all necessary primes $p$, and then suitably glues the cyclotomic algebras associated to these subgroups together.  For a full description of the method, see \cite{OdR}. 

While the ability to find an expression of the Wedderburn decomposition of a group algebra in GAP as a direct sum of matrix rings over cyclotomic algebras is an attractive feature of {\tt wedderga}, the fact that it offered no features for identifying the division algebra parts of the cyclotomic algebras has been an obvious shortcoming.  Users have been left on their own to attempt the delicate calculations of the Schur indices of cyclotomic algebras produced by {\tt wedderga} whenever they need a precise Wedderburn decomposition of the group algebra, such as would be necessary to determine the unit group or automorphism group of $FG$ in a straightforward manner.   

This has been the major motivation for this project, the main part of which implements additional functions in {\tt wedderga} that calculate local Schur indices of cyclotomic algebras.  As a result of the successful implementation of these functions, we are able to provide a Wedderburn decomposition function 

\smallskip
\centerline{ {\tt WedderburnDecompositionWithDivAlgParts(GroupRing(F,G));} }

\smallskip
\noindent that outputs a full Wedderburn decomposition of a group algebra of a finite group $G$ over an abelian number field $F$ in terms of a direct sum of matrices over division algebras.  The additional functions {\tt LocalIndicesOfCyclotomicAlgebra}, {\tt SchurIndex}, and {\tt SchurIndexByCharacter} allow the user to identify the division algebra parts of simple components of group algebras in terms of local indices at rational primes, and to calculate Schur indices of cyclotomic algebras and group characters over any abelian number field.  As part of the same package, we have also provided new functions for calculating Schur indices of quaternion algebras over the field of rational numbers, and tools for converting between quaternion algebras and cyclic algebras.

\medskip
We use the general procedure for calculating local Schur indices of irreducible characters of finite groups that was developed by Bill Unger and Gabriele Nebe in 2006.  Their implementation has been available in MAGMA (Versions 2.14 and up) since 2009 \cite{M}.   For a given finite group $G$ and irreducible character $\chi \in Irr(G)$, it uses the Frobenius-Schur indicator to calculate the local index at $\infty$, and for the local index of $\chi$ at a finite rational prime $q$, it uses these three steps: 

\smallskip
Step 1: (Brauer-Witt search.) For each prime $p$ dividing $\chi(1)$, find a minimal subgroup (i.e. a Schur group) $H$ and $\xi \in Irr(H)$ that isolates the $p$-part of the $q$-local index of $\chi$. 

\smallskip
Step 2: ($q$-modular characters.) If the $q$-defect group of $\xi$ is cyclic, use Benard's theorem on characters in blocks with cyclic defect group \cite{B} to obtain the $q$-local index of $\xi$. 

\smallskip
Step 3: (Dyadic Schur groups.) If the $q$-defect group of $\xi \in Irr(H)$ is not cyclic, then it will be the case that $q=2$,  and one can apply Riese and Schmid's classification of dyadic Schur groups (see \cite{RS} and \cite{S}) to obtain the $2$-local index of $\xi$. 

Aside from these three basic steps, all of the code for the GAP implementation has been created independently.  In addition to characters, it can also accommodate cyclotomic algebras, and makes use of shortcut algorithms for calculating local indices of cyclic cyclotomic algebras.  One of these shortcut procedures is an alternative way to compute the local index of a cyclic cyclotomic algebra at an infinite prime without using the Frobenius-Schur indicator. For the local index of a cyclic cyclotomic algebra at a finite prime, we implement methods due to Janusz \cite{J}.  

The notation used for cyclotomic algebras in {\tt wedderga} will be explained in Section 2, followed by the necessary background on Schur indices of simple components of group algebras in Section 3.  The shortcut algorithms for computing local indices of cyclic cyclotomic algebras are explained in Section 4, along with an explanation of new GAP functions we needed for calculating cyclotomic reciprocity parameters for cyclotomic number fields.  In Section 5, we describe most of the elements for the general procedure.  After showing how we find an irreducible character of a group that realizes a given cyclotomic algebra, we explain our implementation of the Frobenius-Schur indicator for the local index at infinity and of Benard's theorem for the $q$-local index of an irreducible character lying in a block with cyclic $q$-defect group.  These methods leave one exceptional situation, so the final step needed to calculate a $2$-local index using Riese and Schmid's classification of dyadic Schur groups is described in Section 6.    In Section 7, we describe additional features of the package that calculate local indices of rational quaternion algebras using the Legendre symbol procedure, that allow one to decompose a cyclotomic algebra into the tensor product of two cyclic algebras, and that enable the user to convert between cyclic cyclotomic algebras, cyclic algebras, and quaternion algebras.  

The current implementation of the Schur index functions in {\tt wedderga} has been tested and performs adequately on all groups of order up to 511.  We provide descriptions of possible future enhancements to the program that may prove to be necessary for larger groups.   
 
\section{Cyclotomic algebras in {\tt wedderga}} 

Let $F$ be an abelian number field.  This means that $F$ is a subfield of cyclotomic extension $\mathbb{Q}(\zeta_n)$ of $\mathbb{Q}$.   We will write $\zeta_n$ for a primitive complex $n$-th root of unity ({\tt E(n)} in GAP), and $\sigma_b$ for the $\mathbb{Q}$-linear automorphism of $\mathbb{Q}(\zeta_n)$ that sends $\zeta_n$ to $\zeta_n^b$.  This map is denoted by {\tt GaloisCyc(F,b)} in GAP.   Information about the Wedderburn decomposition of the group algebra of a finite group $G$ over $F$ is produced with the {\tt wedderga} commands 

\smallskip
{\tt WedderburnDecompositionInfo(GroupRing(F,G)); } 

\noindent - which gives a list of all the simple components of $FG$, or 

\smallskip
{\tt SimpleAlgebraByCharacterInfo(GroupRing(F,G),chi);} 

\noindent - which gives the particular simple component of $FG$ corresponding to an ordinary irreducible character {\tt chi} of $G$. 

\smallskip
\noindent The individual simple components resulting from these {\tt -Info} functions appear in one of the following forms:  

1) {\tt [r,F]}, which means the ring of $r \times r$ matrices over $F$; 

\smallskip
2) {\tt [r,F,n,[a,b,c]]}, which means the ring of $r \times r$ matrices over the {\it cyclic cyclotomic algebra} $(F(\zeta_n)/F, \sigma_b, \zeta_n^c) := \oplus_i F(\zeta_n) u^i$, where $u$ commutes with elements of $F$, $ u^a = \zeta_n^c$, and $u \zeta_n = \zeta_n^b u;$

\smallskip
3) {\tt [r,F,n,[[a$_1$,b$_1$,c$_1$],[a$_2$,b$_2$,c$_2$]],[[d]]]}, which means the ring of $r \times r$ matrices over the  cyclotomic algebra $(F(\zeta_n)/F, f) = \oplus_i,j F(\zeta_n) u^i v^j$, where $u$ and $v$ commute with $F$,  $ u^{a_1} =\zeta_n^{c_1}$, $u \zeta_n = \zeta_n^{b_1}u$, $v^{a_2} =\zeta_n^{c_2}$, $v \zeta_n = \zeta_n^{b_2} v$,  and $vu = uv \zeta_n^d$; 

\smallskip
4) {\tt [r,F,n,[[a$_1$,b$_1$,c$_1$],[a$_2$,b$_2$,c$_2$],[a$_3$,b$_3$,c$_3$]],[[d$_{12}$,d$_{13}$],[d$_{23}$]]]}, which means the ring of $r \times r$ matrices over the  cyclotomic algebra $(F(\zeta_n)/F, f) = \oplus_i,j F(\zeta_n) u^i v^j w^k$, where $u$, $v$, and $w$ commute with $F$, 
$$\begin{array}{rlrlrl} 
u^{a_1} &=\zeta_n^{c_1}, \qquad & v^{a_2} &=\zeta_n^{c_2}, \qquad & w^{a_3} &=\zeta_n^{c_3}, \\
u \zeta_n &= \zeta_n^{b_1}u, \qquad & v \zeta_n &= \zeta_n^{b_2} v,  \qquad & w \zeta_n &= \zeta_n^{b_3} v,  \\ 
vu &= uv \zeta_n^{d_{12}}, \qquad & wu &=uw \zeta_n^{d_{13}}, \mbox{ and} & wv &=vw \zeta_n^{d_{23}}.  
\end{array}$$
etc.  Cyclotomic algebras whose factor sets require more than 3 generators can (theoretically) be produced by {\tt wedderga}'s {\tt-Info} functions.  These would appear in the form one would expect based on the above pattern.

\section{Schur indices for simple components of group algebras } 

\medskip
Any simple component of the group algebra $FG$ of a finite group $G$ over an abelian number field $F$ will be equal to a principal ideal $FGe$ of $FG$ corresponding to a centrally primitive idempotent $e$.  This idempotent determines a specific Galois conjugacy class of complex irreducible characters of $G$ with $\chi(e) \ne 0$, and since $F$ has characteristic $0$ the center of the simple component is isomorphic to the field of character values $F(\chi)$ for any of these characters $\chi$.  Since $FGe$ is a central simple $F(\chi)$-algebra, it is isomorphic to $M_r(D)$, the ring of $r \times r$ matrices over a division algebra $D$.  We call $D$ the {\it division algebra part} of the simple component.  The {\it Schur index} of the simple component (aka. the Schur index of $D$ or of $\chi$ over $F$) is $m_F(\chi) =\sqrt{[D:F(\chi)]}$, the square root of the dimension of the division algebra $D$ over its center.  So it is an essential invariant for a finite-dimensional division algebra that measures its noncommutativity.
     
It is a consequence of the Brauer-Witt theorem that any simple component of the group ring of a finite group over an abelian number field will be (Morita) equivalent to a cyclotomic algebra over a (possibly larger) abelian number field $F$ (see \cite{Y}).  This means that the simple component $FGe$ corresponding to $\chi \in Irr(G)$ can be expressed as a matrix ring over a cyclotomic algebra with center $F(\chi)$, which the initial Wedderburn decomposition functions in {\tt wedderga} provided.  These cyclotomic algebras are themselves central simple algebras over $F(\chi)$ with the same division algebra part as $FGe$.   

Since its center $F(\chi)$ is an algebraic number field, the isomorphism type of $D$ as an $F(\chi)$-algebra is determined by its list of $\mathcal{P}$-local invariants, one for each prime $\mathcal{P}$ of the number field $F(\chi)$.  Each local invariant is a lowest-terms fraction $s/m_{\mathcal{P}}$ in $\mathbb{Q}$ modulo $\mathbb{Z}$, the local invariant at any infinite prime of $F(\chi)$ can only be $0$ or $\frac12$, all but finitely many of the local invariants are equivalent to $0$, and the sum of all the local invariants also must be $0$ \cite{R}.   The denominator $m_{\mathcal{P}}$ of the local invariant at $\mathcal{P}$ is called the $\mathcal{P}$-local index of $D$.  This is the Schur index of the simple algebra formed by the $\mathcal{P}$-adic completion of $D$.  The global Schur index of $D$ at $F(\chi)$ or $F$ is the least common multiple of these local indices.  

The division algebra parts of simple components of the group rings of finite groups over abelian number fields are restricted quite a bit by the {\it Benard-Schacher theorem}. 

\begin{thm} \cite{BS}  Suppose $D$ is the division algebra part of a simple component of the group algebra of a finite group $G$ over an abelian number field $F$. 

i) As $\mathcal{P}$ runs over the set of primes lying over the same (infinite or finite) rational prime $p$, the local indices $m_{\mathcal{P}}$ for $D$ are all equal to the same positive integer, which we call the {\it $p$-local index} of $D$. 

ii) The pattern of local invariants at the primes lying over a common rational prime is uniformly distributed: 
$$ b \cdot inv_{\mathcal{P}^{\sigma_b}}(D) \equiv inv_{\mathcal{P}}(D) \mod \mathbb{Z}, \forall \sigma_b \in Gal(\mathbb{Q}(\zeta_n)/\mathbb{Q}). $$
(This essentially means by knowing one invariant above $p$ you can determine the others.) 

iii) The Schur index of $D$ is at most the maximum order of a root of unity in $F(\chi)$.   
\end{thm}

There are even further restrictions on the $p$-local indices of division algebra parts of simple components of group algebras.  Especially, the $p$-local index is $1$ for any finite prime $p$ not dividing $|G|$, the $p$-local index is at most $p-1$ for an odd prime $p$, and the $2$-local index can be at most $2$ (see \cite{J} or \cite{Y}). 

\medskip
The goal of the Schur index functions in {\tt wedderga} is to enable the user to identify the division algebra part of a cyclotomic algebra in terms of its $p$-local indices, and from these obtain its Schur index.   Although the list of local indices does not identify the division algebra up to ring isomorphism in general, it does do so when the Schur index is at most $2$ or when the defining group is ``small enough''.  For example, this is true for faithful characters of all groups of order less than 819, since this is the least order of a finite group that can produce $p$-local indices equal to $3$ at two distinct odd primes $p$.   

\section{Shortcuts for Schur indices of cyclic cyclotomic algebras} 

When a simple component produced by {\tt wedderga} has the presentation 

\smallskip
\centerline{ {\tt [r,F,n,[a,b,c]]},}

\smallskip
\noindent it is a matrix ring over the {\it cyclic cyclotomic algebra} $A=(F(\zeta_n)/F, \sigma_b, \zeta_n^c)$.   The local indices of these functions can be calculated directly from this presentation using three shortcut algorithms. 

\medskip
\noindent 1. {\tt LocalIndexAtInfty}.  A shortcut for calculating the local index of $A$ at an infinite prime is given by the following lemma.

\begin{lemma} 
A cyclic cyclotomic algebra $A=(F(\zeta_n)/F, \sigma_b, \zeta_n^c)$ has local index $2$ at an infinite prime if and only if $F \subset \mathbb{R}$, $n>2$, and $\zeta_n^c=-1$. 
\end{lemma} 

\begin{proof} 
$F$ must be a real subfield, for otherwise the center of the completion of $A$ at an infinite prime would be $\mathbb{C}$ and no Schur index other than $1$ would be possible.  When $F$ is a real subfield and $n>2$, $[F(\zeta_n):F]=a$ is even, and  $\sigma_b^{a/2}$ will induce the nontrivial automorphism of order $2$ in $Gal(F(\zeta_n)/(F(\zeta_n) \cap \mathbb{R}))$.  This will induce complex conjugation in $\mathbb{C}/\mathbb{R}$, so $A \otimes \mathbb{R}$ will be Morita equivalent to  $(\mathbb{C}/\mathbb{R}, \sigma_b^{a/2}, \zeta_n^c)$.  Since $\zeta_n^c$ has to be a root of unity in $\mathbb{R}$, this algebra will have index $2$ only when $\zeta_n^c=-1$.
\end{proof} 

Simply checking these conditions allows us to calculate the real Schur index of a cyclic cyclotomic algebra without using the Frobenius-Schur indicator of a irreducible character, which would require obtaining a faithful character of a group that can be represented by the cyclotomic algebra $A$ and thus be a bit more expensive.  Functions that calculate the local index using the traditional character-theoretic approach are described in sections 5 and 6.  These can be applied to cyclic cyclotomic algebras directly, which makes it possible to check results obtained using the shortcut methods.   

\medskip 
\noindent 2. {\tt LocalIndexAtOddP}.  The shortcut for the local index of $A=(F(\zeta_n)/F, \sigma_b, \zeta_n^c)$ at an odd prime $p$ makes use of the following lemma of Janusz. 

\begin{lemma}[\cite{J}, Lemma 3.1]
Let $E$ be a Galois extension of a local field $K$ with ramification index $e=e(E/K)$.  Suppose $n$ is relatively prime to $p$ with $\zeta_n \in K$.  Then $\zeta_n$ is a norm in $E/K$ if and only if $\zeta_n$ is the $e$-th power of a root of unity in $K$.  
\end{lemma} 

This lemma means that one can calculate the $p$-local index of $A$ by counting the roots of unity in a $p$-adic completion of $F$.  To do this, we use the fact that the number of roots of unity of order coprime to $p$ in a $p$-adic completion of $F$ is $p^f-1$, where $f=f(F/\mathbb{Q})$ is the residue degree of $F/\mathbb{Q}$ at the prime $p$.  In order to use these facts, we have implemented (new!) cyclotomic reciprocity calculations in GAP.  These enable us to find, for $F \subseteq K \subseteq \mathbb{Q}(\zeta_n)$, the ramification index $e=e(K/F)$, the residue degree $f=f(K/F)$, and the splitting degree $g=g(K/F)$ at the prime $p$.   Once we have found $e=e(F(\zeta_n)/F)$ and $f=f(F/\mathbb{Q})$, it is immediate from Janusz's Lemma that the $p$-local index of $A = (F(\zeta_n)/F, \sigma_b, \zeta_n^c)$ for an odd prime $p$ is simply the least power of $\zeta_n^c$ that is in the group generated by $\zeta_{p^f-1}^e$.  

\medskip
\noindent {\bf Cyclotomic reciprocity in GAP.} Our cyclotomic reciprocity calculations begin with the calculation of $e$, $f$, and $g$ for the extension $\mathbb{Q}(\zeta_n)/F$, including the case $F = \mathbb{Q}$.  Let $n_{p'}$ be the $p'$-part of $n$, the largest divisor of $n$ that is coprime to $p$, and let $n_p$ be the largest power of $p$ dividing $n$.   The Galois group of $\mathbb{Q}(\zeta_n)/\mathbb{Q}$ is the set of all $\sigma_b$ with $b$ coprime to $n$, and $Gal(\mathbb{Q}(\zeta_n)/F)$ is the subgroup consisting of those $\sigma_b$'s that fix a primitive element of $F$.   Let $B$ be the set of integers $b$ modulo $n$ corresponding to the $\sigma_b$'s in $Gal(\mathbb{Q}(\zeta_n)/F)$.  Reduce the elements of the set $B$ modulo $n_{p'}$ to get the set $\bar{B}$.  The size of this set is $fg$, and $|B|/|\bar{B}| = e$.   Next, let $U$ be the set of distinct powers of $p$ modulo $n_{p'}$.  The size of the intersection of $U$ and $\bar{B}$ is $f$.  The subfield of $F(\zeta_{n_{p'}})$ fixed by these $\sigma_b$'s is the maximal subextension of $\mathbb{Q}(\zeta_n)/F$ that is totally split at the prime $p$.  Its dimension over $F$ is $g$.  (As this field is needed for the general algorithm later on, the command {\tt PSplitSubextension(F,n,p)} produces it directly.)  

Once we calculate the field parameters $e$, $f$, and $g$ for the extension $\mathbb{Q}(\zeta_n)/\mathbb{Q}$, we then calculate the $e'$, $f'$, and $g'$ for the extension $\mathbb{Q}(\zeta_n)/F$.  Since these are abelian Galois extensions of $\mathbb{Q}$, the parameters $e''$, $f''$, and $g''$ for $F/\mathbb{Q}$ are then just the ratios $e'' = e/e'$, $f''=f/f'$, and $g''=g/g'$.   The parameters for general extensions $K/F$ for $K \subseteq \mathbb{Q}(\zeta_n)$ are obtained similarly from the parameters of $\mathbb{Q}(\zeta_n)/K$ and $\mathbb{Q}(\zeta_n)/F$.

\medskip
\noindent 3. {\tt LocalIndexAtTwo.}  The shortcut algorithm for the $2$-local index of cyclic cyclotomic algebras is also based on the ideas in \cite{J}.  It requires the following observation.

\begin{lemma} 
Suppose the cyclic cyclotomic algebra $A = (F(\zeta_n)/F,\sigma_b,\zeta_n^c)$ over an abelian number field $F$ has $2$-local index $2$.  Then any $2$-adic completion of $A$ is Morita equivalent to a non-split quaternion algebra over a $2$-adic completion of $F$. 
\end{lemma}

\begin{proof} 
Let $K$ be the completion of $F$ at a prime of $F$ lying over $2$, and let $L=K(\zeta_n)$.  Then $A_K = A \otimes_F K \simeq (L/K, \sigma_b^g, \zeta_n^c)$ is a cyclic algebra over a $2$-local field with index $2$.  If it were the case that $4$ did not divide $n$, then $L/K$ would be unramified, forcing $\zeta_n^c$ to be a norm in $L/K$ and $A_K$ to have index $1$. So it must be the case that $4$ divides $n$.  

Find primitive $n_2$ and $n_{2'}$ roots of unity for which $\zeta_n = \zeta_{n_2} \zeta_{n_{2'}}$.  We can decompose $A_K$ as the tensor product 
$$ A_K \simeq (L/K, \sigma_b^g, \zeta_{n_2}^c) \otimes_K (L/K, \sigma_b^g, \zeta_{n_{2'}}^c),$$
and the second factor must be split because its factor set has odd order and its local index must divide $2$.  So $A_K$ is equivalent to a cyclic cyclotomic algebra whose factor set is contained in $2$-power roots of unity.  

By a theorem of Witt (see \cite[Proposition 3.2]{J}), $\zeta_4 \not\in K$.  Since $\zeta_{n_2}^c \in K$, in order for the index of $A_K$ to be $2$ we must have $\zeta_{n_2}^c = -1$.   

If $F$ is the maximal subextension of $L$ containing $K$ such that $[L:F] = \ell$ is odd, then by \cite[Theorem (30.10)]{R}, 
$$ (L/K, \sigma,-1) \simeq (F/K,\bar{\sigma},-1). $$
Now, if it were the case that $K(\zeta_4)$ would be properly contained in $F$, then there would be a subfield $F_1$ of $F$ containing $K$ that is linearly disjoint from $K(\zeta_4)$.  Since $[F_1:K]$ would be a power of $2$, the Galois group of $[F:K]$ would not be cyclic, a contradiction.   Therefore, $A_K$ is equivalent to a nonsplit quaternion algebra $(K(\zeta_4)/K,\bar{\sigma},-1)$.  
\end{proof} 

For the shortcut algorithm for the $2$-local index of a cyclic cyclotomic algebra $A=(F(\zeta_n)/F,\sigma_b,\zeta_n^c)$, we use these steps: 

\smallskip
\noindent Step 1:  find the maximal $2$-split subextension $K$ of $F(\zeta_n)/F$ and determine $A_K = (F(\zeta_n)/K,\sigma_b^g,\zeta_n^c)$; 

\smallskip
\noindent Step 2: arrange for $n$ (and $c$) to be minimal with $A_K = (K(\zeta_n)/K,\sigma_b^g,\zeta_n^c)$; 

\smallskip
\noindent Step 3: check that $4$ divides $n$, $\zeta_4 \not\in K$, and the order of $\zeta_n^c$ is twice an odd number; and

\smallskip
\noindent Step 4: show that the residue degree $f = f(K/\mathbb{Q})$ is odd (otherwise the quaternion algebra over $K$ would split at the prime $2$).

%\smallskip
%\noindent Step 5: check that the index of $K(\zeta_4)$ in $K(\zeta_n)$ is odd.  

\smallskip
\noindent The $2$-local index will be $2$ only if the conditions of steps 3 and 4 are satisfied.  

\section{Computing local indices using ordinary and modular characters}

Since the shortcuts in the previous section can only be applied to cyclic cyclotomic algebras, a general method must be applied to non-cyclic cyclotomic algebras, which {\tt wedderga}'s {\tt -Info} functions represent as lists of length 5.  For these algebras, a traditional approach using character-theoretic information is effective.   The theory behind these algorithms is fairly well-established.   

Let $A$ be a cyclotomic algebra with center an abelian number field $F$.   One must first use the presentation of the cyclotomic algebra to extract a faithful irreducible character of a group that naturally defines it.  Then we apply the character-theoretic methods to this group and character.  In most cases we are thinking of, the cyclotomic algebra $A$ will not be cyclic cyclotomic, but if it is these functions can still be applied.  

\smallskip
\noindent 1. {\tt DefiningGroupOfCyclotomicAlgebra.}  This constructs a natural quotient of a free group directly from the cyclotomic algebra presentation of $A$.  For example, if 
$A = ${\tt [r,F,n,[[a$_1$,b$_1$,c$_1$],[a$_2$,b$_2$,c$_2$]],[[d]]]} then the defining group $G$ is the quotient group of the free group on 3 generators $x$, $y_1$, and $y_2$ defined using these relations: 

$$\begin{array}{rclrclrcl} 
x^n &=&1, & {y_1}^{a_1} &=& x^{c_1}, & {y_2}^{a_2} &=& x^{c_2}, \\
x^{y_1} &=& x^{b_1}, & x^{y_2} &=& x^{b_2}, & y_2y_1 &=& y_1y_2x^d. 
\end{array}$$

To expedite calculations later on, we immediately determine a polycyclic presentation of $G$.  This is done using the GAP command  {\tt IsomorphismSpecialPcGroup}.   It follows naturally from the cyclotomic algebra presentation of $A$ that these defining groups are always cyclic-by-abelian, and therefore GAP can find their polycyclic representations effectively.  The polycyclic presentation of $G$ aids with the calculation of its character table.    

\medskip
\noindent 2. {\tt DefiningCharacterOfCyclotomicAlgebra.}  Let the defining group of $A$ be $G$.  The defining character will be a faithful irreducible character $\chi$ of $G$ for which the simple component of $FG$ corresponding to $\chi$ is $A$.  The function identifies the character $\chi$ simply by returning an integer $s$ for which the GAP character ${\tt Irr(G)[s]}$ is one of the Galois conjugates of $\chi$.  

\medskip
\noindent {\bf Remark.} Even when {\tt Irr(G)[s]} is the defining character of $A$, the presentation returned by a new call of 

\smallskip
\centerline{{\tt SimpleAlgebraByCharacterInfo(GroupRing(F,G),Irr(G)[s])}}

\smallskip
\noindent will often not match that of the original presentation of $A$ exactly because of the randomised methods that may be used by GAP to calculate $Irr(G)$.   However, the Brauer-Witt approach used by {\tt wedderga} ensures that the cyclic-by-abelian defining group appearing in the construction is a minimal one necessary to obtain a presentation of the simple component over the field $F$.   This minimality ensures that there is only one Galois conjugacy class of faithful irreducible characters of $G$ whose simple component is realized as a crossed product of the extension $F(\zeta_n)/F(\chi)$ being acted on by its Galois group, and thus good candidates for $s$ are easily found with search of the character table of $G$.   To see that there is only one Galois conjugacy class of these characters, note that every faithful irreducible character $\chi$ of $G$ is induced from a faithful irreducible character $\lambda$ of a maximal cyclic normal subgroup $C$ of $G$.  If $\phi$ is any other faithful character of $G$, then $\phi$ also is induced from a faithful irreducible character of $C$, which must be equal to $\lambda^\sigma$, for some $\sigma \in Gal(F(\zeta_n)/F)$.  It then follows that $\phi^\sigma=\chi$. 

\medskip
\noindent 3. {\tt LocalIndexAtInftyByCharacter.}  The standard method for calculating the local index at $\infty$ of the simple component of $\mathbb{Q}G$ corresponding to $\chi$ is to use the Frobenius-Schur indicator of $\chi$.  If $\chi$ is represented in GAP by the character {\tt Irr(G)[s]}, then the Frobenius-Schur indicator of $\chi$ is the result of 

\smallskip
\centerline{{\tt Indicator(CharacterTable(G),2)[s]}.}

\smallskip
\noindent The local index of $\chi$ at $\infty$ is $2$ exactly when this value is $-1$. 

\medskip
\noindent 4. {\tt LocalIndexAtPByBrauerCharacter.}  Let $G$ be the defining group and $\chi$ be the defining character of $A$.  If $G$ happens to be nilpotent, it is well-known that the only possibilities for a simple component of $G$ to have non-trivial division algebra part occur when the component is equivalent to the ordinary quaternion algebra.  Therefore, these cases will result in cyclic cyclotomic algebras that can be handled by the shortcut method.  If $G$ is not nilpotent, then it will still be cyclic-by-abelian, with its order divisible by at least two primes.   Brauer characters of such groups are quite accessible by means of the Fong-Swan theorem.  Indeed, if the $p$-defect group of the block containing $\chi$ is cyclic, we can apply the following theorem of Benard.

\begin{thm}[\cite{B}] 
Suppose $\chi \in Irr(G)$ lies in a $p$-block $B$ of $G$ for which the $p$-defect group of $B$ is cyclic.  Let $\chi_o$ be the restriction of $\chi$ to the set of $p$-regular elements of $G$, and let $\phi$ be an irreducible Brauer character lying in the block $B$.  

Then the $p$-local index of $\chi$ is $[\mathbb{Q}_p(\chi_o,\phi) : \mathbb{Q}_p(\chi_o)]$.    
\end{thm}

Provided the order of $G$ is small enough, it is straightforward to use GAP's Brauer character records to find the values of an irreducible Brauer character $\phi$ of $G$ lying in the same $p$-block of $G$.  To use Benard's theorem, we convert the values of $\chi_o$ and $\phi$ to lie in a finite field extension.  This relies on fixing an isomorphism from the group of $p$-regular roots of unity in $\mathbb{Q}(\zeta_n)$ into the multiplicative group of a finite field of characteristic $p$.   

Benard's theorem only gives correct results when the field $F$ is contained in the field of character values. To compute the $p$-local index over an abelian number field $F$ for which $\mathbb{Q}(\chi) \subsetneq F(\chi)$, we apply a theorem of Yamada (see \cite[Theorem 9.2]{F}).

\begin{thm}[\cite{Y}]
Let $K$ be a finite extension of $\mathbb{Q}_p$, and let $\chi \in Irr(G)$.  If $L$ is a finite extension of $K(\chi)$ for which $m_{K}(\chi) | [L:K(\chi)]$, then $m_L(\chi) = 1$.
\end{thm} 

To apply this theorem in cases where $F$ is larger than the field of character values, we need to find the degree $d$ of the $p$-local extension corresponding to the global extension $F/\mathbb{Q}(\chi)$.  This will be $d = e(F/\mathbb{Q}(\chi),p) f(F/\mathbb{Q}(\chi),p)$, which we find using our cyclotomic reciprocity functions.   Then we divide the $p$-local index of $\chi$ over the field of character values by $gcd(m_{\mathbb{Q}_p}(\chi),d)$ to get the $p$-local index of $\chi$ over $F$. 

\smallskip
Although the $p$-defect group $D$ of the $p$-block of $\chi = ${\tt Irr(G)[s]} can not yet be calculated precisely in GAP, it is well-known that 

\smallskip
- $D$ is the intersection of a Sylow $p$-subgroup $P$ of $G$ with its conjugate by a $p$-regular element of $G$; and 

\smallskip
- $D$ is the Sylow $p$-subgroup of the centralizer in $G$ of a $p$-regular element of $G$.  

\smallskip
\noindent The command {\tt PossibleDefectGroups(G,s,p)} determines a list of conjugacy classes of the $p$-subgroups of $G$ satisfying both of these conditions.  (We thank Michael Geline for suggesting this practical approach to the calculation of defect groups in GAP.)   We then check if representatives of all of these conjugacy classes are cyclic, and if so we proceed to compute the $p$-local index by the Brauer character method.  

If the $p$-defect group $D$ is not, or may not be, cyclic, the Brauer character calculation of the $p$-local index is unreliable.  (An exception is a recent theorem of Geline \cite{G}, which shows  $m_{\mathbb{Q}_2}(\chi)=1$ when $G$ is solvable and the $2$-defect group of the block containing $\chi$ is abelian.)   When $p$ is odd, we can avoid this situation by calculating the $p$-local index one prime part at a time.  To calculate the largest power of $q$ that divides the $p$-local index of $\chi \in Irr(G)$, we first find $A_K$ where $K$ is the maximal $p$-split subextension $K$ of $L=F(\zeta_n)$.  For each prime $q$ dividing both $|G|$ and $p-1$, find the unique subfield $K_q$ of $L$ containing $K$ for which $[K_q:K]$ is coprime to $q$ and $[L:K_q]$ is a $q$-power.  The $p$-local index of $A_{K_q}$ is equal to the $q$-part of the $p$-local index of $A$, and the defining group of $A_{K_q}$ has a cyclic Sylow $p$-subgroup.  

\medskip
\noindent {\bf Remark.}  As it has not been required in the computation of Schur indices of groups of order up to 511, a command implementing the ``one prime part at a time'' algorithm to calculate the $p$-local index of $\chi$ has yet to be developed in {\tt wedderga}.   The steps necessary for this reduction can be manually performed one at a time in the current version, however. 

\section{$2$-local indices via the classification of dyadic Schur groups} 

\noindent {1. {\tt LocalIndexAtTwoByCharacter.}} When $p = 2$, the $2$-local index of the cyclotomic algebra $A = (K(\zeta_n)/K,\alpha)$ can only divide $2$, so we can again arrange that the factor set of $A$ consists only of $2$-power roots of unity.  As above we find the maximal $2$-split subextension $K$ of $F$, then the unique subfield $K_2$ of $L=K(\zeta_n)$ containing $K$ for which $[K_2:K]$ is odd and $[L:K_2]$ is a power of $2$.  If possible, we then replace $n$ with a $n'<n$ with $K_2(\zeta_{n'})=L$ , and re-calculate {\tt wedderga}'s {\tt-Info} presentation of $A_{K_2}$ so that it uses a primitive $n'$-th root of 1.  As discussed earlier, the $2$-local index of $A_{K_2}$ cannot be $2$ unless $4$ divides $n'$ and $\zeta_4 \not\in K_2$.   The defining group of $A_{K_2}$ is an extension of a cyclic group by an abelian $2$-group with at most $3$ generators, at most one for the unramified part of the corresponding $2$-local extension, and at most two for the ramified part.   There is an automorphism $\sigma_b \in Gal(L/K_2)$ which inverts $\zeta_4$ and fixes $\zeta_{n'_{2'}}$.  If there are $2$ generators for the ramified part, then $n'_2>4$ and the other generator would have to be an automorphism of $L/K_2$ fixing $\zeta_{4n'_{2'}}$ for which $\zeta_{n'_2} \mapsto \zeta_{n'_2}^{5^k}$ for some integer $k$.  If $E$ is the subfield of $L$ fixed by this automorphism, then the norm of the extension $L/E$ will map $\langle \zeta_{n'_2} \rangle$ surjectively onto $\langle \zeta_4 \rangle$ (see \cite[Theorem 1]{J}).   The field $E$ is now of the form $E=K_2(\zeta_{4n_{2'}})$.  Since the residue degree of $E/K_2$ is a power of $2$, by minimality we can assume $n_{2'}$ is a prime $q$.   It then follows from \cite[Theorem 1]{J} that the cyclotomic algebra $A_{K_2}$ is equivalent to a cyclotomic algebra of the form $(E/K_2,\alpha)$, with the factor set $\alpha$ taking values in $\langle \zeta_{4} \rangle$.   

\medskip
\noindent {\bf Remark.}  As it has not been required to calculate Schur indices for groups of order up to 511, a subroutine to carry out the norm reduction from $L/K_2$ to $E/K_2$ described above has yet to be implemented in the current version of {\tt wedderga}.   In fact it may be the case that the bottom-up nature of the search routines used in the Brauer-Witt reduction make this step unnecessary in practice - the cyclotomic algebra that {\tt wedderga} obtains over the maximal $2$-split subextension may only have $1$- or $2$-generated Galois groups.  This is the case if one uses a ``one prime at a time" strategy in the traditional Brauer-Witt theorem, but the cyclotomic algebras produced by {\tt wedderga} are not exactly done with this method.  If $3$-generated galois groups over the $2$-split subextension do occur, this issue will have to be dealt with in a future version of the package.   With the current package, a manual implementation of the ``one prime at a time" strategy can be carried out, so calculation of the $2$-local index, while not automatic, would still be possible.     

\smallskip
Even when we have reduced the $2$-local index calculation of $\chi$ to that of $(E/K_2,\alpha)$, it is often still not the case that the defect group of the $2$-block containing the defining character is cyclic, or even abelian.  When this defect group is abelian, the theorem of Geline says the $2$-local index of $\chi$ is $1$.  If the defect group is not guaranteed to be abelian, we appeal to the classification of dyadic Schur groups, which we now describe. 

\medskip
\noindent 2. {\tt IsDyadicSchurGroup.}  This process should only be used when $(E/K_2,\alpha)$ is not a cyclic cyclotomic algebra, and it should be the case that $E=K_2(\zeta_{4q})$ for some odd prime $q$, $\zeta_4 \not\in K_2$, $|Gal(E/K_2)|$ is a power of $2$, $\alpha \subseteq \langle \zeta_4 \rangle$, and the defect group of the corresponding $2$-block is not guaranteed to be abelian.   In this case the defining group and defining character for $(E/K_2,\alpha)$ coincide with a terminal Brauer-Witt reduction (as explained in the next paragraph) for the character $\chi$ in the terminology of \cite{S}.  By \cite{S} and \cite{RS}, the $2$-local index of $(E/K_2,\alpha)$ will be $2$ in this case if and only if the structure of the defining group matches one of two types of {\it dyadic Schur groups} whose faithful characters lie in $2$-blocks with nonabelian defect groups, and the $2$-local index of the defining character remains nontrivial over $K_2$.   

Let $p$ be a prime dividing $|G|$ and $q$ be a prime dividing $\chi(1)$. A {\it Brauer-Witt reduction} for the $q$-part of the $p$-local index of $\chi \in Irr(G)$ is a pair $(H,\xi)$ formed by an irreducible character $\xi$ of a subgroup $H$ of $G$, with the property that both $(\chi_H,\xi)$ and $[\mathbb{Q}_p(\chi,\xi):\mathbb{Q}_p(\chi)]$ are not divisible by $q$.  These conditions ensure that the $p$-local index of $\xi$ is equal to the $q$-part of the $p$-local index of $\chi$.   The pair $(H,\xi)$ is a {\it terminal Brauer-Witt reduction} for (the $q$-part of the $p$-local index of) $\chi$ when $(H,\xi)$ is a Brauer-Witt reduction, but no proper subgroup of $H$ can be used in a Brauer-Witt reduction for $\xi$.  A {\it Schur group} is a group $H$ with a faithful irreducible character $\xi$ for which there is a group $G$ and character $\chi$ of $G$ for which $(H,\xi)$ is the terminal Brauer-Witt reduction for the $q$-part of the $p$-local index of $\chi$.  The main results of \cite{S} show that for every irreducible character $\xi$ of a finite group $G$, and for every suitable pair of primes $p$ and $q$, there is a subgroup $H$ and $\xi \in Irr(H)$ for which $(H,\xi)$ is a terminal Brauer-Witt reduction for the $q$-part of the $p$-local index of $\chi$.  Schmid organized Schur groups into 7 different structural types, and for each of these types, gave a formula $p$-local index of its unique faithful irreducible character $\xi$.  The {\it dyadic Schur groups} are the Schur groups $H$ for which the $2$-local index of their faithful irreducible character $\xi$ is $2$.  

We will use Riese and Schmid's characterization of dyadic Schur groups from \cite{RS}, which requires a careful definition of dyadic $2$-groups.  We say that a $2$-group $P$ is {\it dyadic} if $P'$ is cyclic of order at least $4$ and the centralizer $Y$ of the subgroup $Z$ of order $4$ in $P'$ has the property that $Y/Z$ is cyclic \cite[Lemma 3]{RS}.   

\begin{thm}[\cite{RS}, Lemma 4]
Suppose $H$ is a dyadic Schur group and $H \ne Q_8$.   Then $H \simeq U \rtimes P$, where $U \simeq C_q$ is cyclic of prime order $q$, and $P$ is a $2$-group.  Let $X=C_P(U)$.  Then one of the following holds: 

a) $H$ is of type $(Q_8,q)$: $X \simeq Q_8$, and $P$ is the central product of $X$ and $C_P(X)$; or 

b) $H$ is of type $(QD,q)$: $X$ is a generalized quaternion $2$-group of order $\ge 16$ or a dihedral $2$-group, and $P$ is a dyadic $2$-group with $|X/P'| = 2$.  
\end{thm}  

\noindent {\bf Remark.} The condition that $X$ not be allowed to be $Q_8$ for $H$ to have type $(QD,q)$ can be inferred from a careful reading of the proof of \cite[Lemma 4]{RS}.  

\medskip
If the defining group $H$ is of type $(Q_8,q)$, then the simple component generated in {\tt wedderga} is actually a cyclic cyclotomic algebra, since $H \simeq C_{4q} : C_{2^s}$ if the action of $P$ on $U$ has order $2^s$.  (We use the notation $X:H$ to indicate a non-split extension of $X$ by $H$.)  So in these cases we can compute the $2$-local index directly once we have reduced to the field $E$.  If the defiining group $H$ is of type $(QD,q)$,  we check that its Sylow $2$-subgroup has the structure of a dyadic $2$-group, then we check the conditions that $X=C_P(U)$ is generalized quaternion or dihedral and that $X/P'$ has order $2$.  

Finally we have to check that the field $K_2$ does not split the algebra.  This step is not needed when computing indices of group algebras over the rationals, but it has to be dealt with when working over abelian number fields that are larger than the field of character values.  As we had shown in the discussion concerning Benard's theorem, this will be the case if and only if $e(K_2/\mathbb{Q}(\chi),2)f(K_2/\mathbb{Q}(\chi),2)$ is odd. 

\medskip
\noindent {\bf Example.} Nontrivial dyadic Schur groups of type $(QD,q)$ already need to be considered for groups of order $48$.  For example, consider the small groups of order 48 numbered 15 through 18 in the GAP library.  These each have a single faithful rational-valued irreducible character $\chi$ of degree $4$.  The local indices of these characters is as shown in the following table:   
 
$$ \begin{array}{ccccc}
\mbox{ Group }  & \qquad P \qquad & \qquad X \qquad & \qquad P' \qquad & \mbox{ Local indices of } \chi \\ \hline
(48,15) & X \rtimes C_2 & D_8 & C_4 & \mathtt{[[2,2],[3,2]]} \\
(48,16) & X:C_2 & D_8 & C_4  & \mathtt{[[2,2],[\infty,2]]} \\
(48,17) & X \rtimes C_2 & Q_8 & C_4 & \mathtt{[ \quad ]} \\
(48,18) & X:C_2 & Q_8 & C_4 & \mathtt{[[3,2],[\infty,2]]}
\end{array}$$

\noindent The first two of these are dyadic Schur groups of type $(QD,q)$, the second pair are not dyadic Schur groups.  Local indices for these characters can also be checked using the functions provided in the next section. 

\section{Additional tools for cyclic and quaternion algebras} 

\noindent 1. {\tt LocalIndicesOfRationalQuaternionAlgebra.}  GAP 4 includes built-in functions for several special kinds of algebras, including generalized quaternion algebras $(\frac{a,b}{K})$ over abelian number fields $K$, which are entered as {\tt QuaternionAlgebra(K,a,b)} for $a$, $b \in K^{\times}$.  This is the $4$-dimensional central simple algebra over $K$ generated by elements $i$, $j$ with the relations $i^2=a$, $j^2=b$, and $ij = -ji$.   

It is often desirable to know when a generalized quaternion algebra is a division algebra, i.e.~when its Schur index is $2$.  For generalized quaternion algebras over $\mathbb{Q}$, there is a traditional  algorithm for calcuating their $p$-local indices  that makes use of the Legendre symbol.  It is outlined in the book of Pierce \cite[page 366]{P}.  One first obtains a tensor factorization of $(\frac{a,b}{\mathbb{Q}})$ into quaternion algebra factors of the form $(\frac{c,d}{\mathbb{Q}})$ where both $c$ and $d$ are either prime, $1$, or $-1$.   The Legendre symbol at $p$ is then used to directly compute the $p$-local invariants of these algebras.  Since the only nonzero invariants are $\frac12$'s, it is easy to sum them up modulo $\mathbb{Z}$ to get the $p$-local invariant of $(\frac{a,b}{\mathbb{Q}})$ and hence the $p$-local index.  The local index at $\infty$ is $2$ only if both $a$ and $b$ are negative.   If $A$ is a rational quaternion algebra in GAP, {\tt LocalIndicesOfRationalQuaternionAlgebra(A)} returns its list of local indices, and {\tt SchurIndex(A)} returns its rational Schur index.  These functions will only work if the center of the quaternion algebra $A$ is precisely the rationals, and otherwise result in an error.   We have found that using our {\tt SchurIndex(A)} is to be more reliable than GAP's {\tt IsDivisionRing(A)} for deciding when a generalized quaternion algebra over $\mathbb{Q}$ is a division algebra.  

\medskip
{\bf Cyclic algebras and Quadratic algebras.} For quaternion algebras over abelian number fields larger than $K$, we offer functions that convert them into cyclic algebras, whose local indices can be computed either by solving suitable norm equations, or in some special cases by further conversion into a cyclic cyclotomic algebra.   In {\tt wedderga}, we represent the cyclic algebra $(L/K,\sigma,a)$ as simply {\tt [K,L,[a]]}.  A well-known result states that the Schur index of the cyclic algebra $(L/K,\sigma,a)$ is the least power $m$ of $a$ that lies in the image of the norm map $N_{L/K}$.  While GAP offers no command for solving such norm equations, other number theory systems do offer this capability.  For example, one can use  the command {\tt bnfisnorm(F,x,\{flag=0\})} in PARI/GP \cite{PARI2} to verify whether or not the nonzero rational number $x$ is a norm in the Galois extension $F/\mathbb{Q}$.  (The referee has remarked that the PARI/GP interface provided by GAP's {\tt Alnuth} package might be utilized for future improvements to {\tt wedderga}.)

When $L = K(\sqrt{d})$ for some $d \in K$, we refer to the cyclic algebra as a quadratic algebra.  It is easy to convert directly between quaternion algebras and quadratic algebras: $(K(\sqrt{d})/K,\sigma,a)$ is equal to $(\frac{d,a}{K})$, and vice versa.  If the cyclic algebra $(L/K,\sigma,a)$ has the form $(F(\zeta_n)/F,\sigma_b,\zeta_n^c)$, then we can convert it directly into the cyclic cyclotomic algebra with {\tt wedderga} presentation {\tt [1,F,n,[$[F(\zeta_n):F]$,b,c]]}.  We provide functions in {\tt wedderga} for all of these conversions.  It would also be desirable to be able to convert between our algebras and algebras with structure constants, or finitely presented algebras.  Such functions will be considered for a future release of the package.  

\medskip
\noindent 2. {\tt DecomposeCyclotomicAlgebra}.  Another way to attack the problem of calculating Schur indices for non-cyclic cyclotomic algebras resulting from {\tt wedderga}'s {\tt -Info} functions is decompose them as tensor products of cyclic algebras, whose local indices can be calculated using norm equations or other methods.  Given a cyclotomic algebra {\tt [r,F,n,[[a$_1$,b$_1$,c$_1$],[a$_2$,b$_2$,c$_2$]],[[d]]]}, whose abelian galois group has $2$ generators, we think of the algebra as $\oplus_{i,j} L u^i v^j$, where $u$ and $v$ satisfy the relations given in section 2.    Let $L = F(\zeta_n)$.  One way to decompose the algebra as a tensor product of cyclic algebras is to find a nonzero scalar $c$ in $L^{\times}$ for which $v(cu) = (cu)v$ and $(cu)^{a_1} \in F$.   Once this is found, the cyclotomic algebra decomposes as $r \times r$ matrices over 
$$ (L^{\langle \sigma_{b_2} \rangle}/F, \sigma, (cu)^{a_1}) \otimes_F (L^{\langle \sigma_{b_1} \rangle}/F, \phi, v^{a_2}). $$
Alternatively, we can leave $u$ alone and replace $v$ by a suitable scalar in $L^{\times}$, or even adjust both by a scalar simultaneously.   
The fixed fields $L^{\langle \sigma_{b_i} \rangle}$ are easily found by using the command $NF(n,[1,b_i])$ in GAP when $L = \mathbb{Q}(\zeta_n)$, and if $L$ is otherwise we find $NF(n,[1,b_i])$ and extend it by a primitive element of $F$. 

The special cases where $\zeta_n^d \in \langle \zeta_4 \rangle$ and $\zeta_4 \not\in F$ occur most often for small groups.  To achieve the decomposition in these cases, we replace $u$ and $v$ by the $u'$ and $v'$ in the following table:  

$$\begin{array}{ccccc}
\zeta_n^d & \zeta_4^{b_1} & \zeta_4^{b_2}& u' & v' \\ \hline
1 & * & * & u & v \\
-1 & \pm \zeta_4 & \zeta_4 & u & \zeta_4 v \\
-1 & \pm \zeta_4 & -\zeta_4 & \zeta_4 u & v \\
\zeta_4 & \zeta_4 & -\zeta_4 & (1-\zeta_4)u & v \\
\zeta_4 & -\zeta_4 & \zeta_4 & u & (1+\zeta_4) v \\
\zeta_4 & -\zeta_4 & -\zeta_4 & (1-\zeta_4)u & \zeta_4 v \\
-\zeta_4 & \zeta_4 & -\zeta_4 & (1+\zeta_4)u & v \\
-\zeta_4 & -\zeta_4 & \zeta_4 & u & (1 - \zeta_4) v \\
-\zeta_4 & -\zeta_4 & -\zeta_4 & (1+\zeta_4)u & \zeta_4 v 
\end{array}$$

\smallskip
\noindent {\bf Example.}  An example of the above is provided by the decomposition of the cyclotomic algebras resulting from the small groups of order 48 numbered 15 to 18 in the GAP library.  Each of these groups $G$ produces a simple component in $\mathbb{Q}G$ that is a non-cyclic cycoltomic algebra.  After loading {\tt wedderga}, the reader is invited to generate these examples for themselves using these commands 

\medskip
{\tt gap> R:=GroupRing(Rationals,SmallGroup(48,15));;}

{\tt gap> W:=WedderburnDecompositionInfo(R);;}

{\tt gap> DecomposeCyclotomicAlgebra(W[10]);}

\smallskip
\noindent The following table gives the cyclic algebra decompositions of these four components. 

$$\begin{array}{ccc} 
\mbox{Group} & \mbox{Simple Component} & Tensor Factors  \\ \hline
(48,15) & \mathtt{[1,Rationals,12,[[2,5,9],[2,7,0]],[[9]]]} & (\mathbb{Q}(\zeta_3)/\mathbb{Q},2) \otimes (\mathbb{Q}(\zeta_4)/\mathbb{Q}, 1) \\
(48,16) & \mathtt{[1,Rationals,12,[[2,5,3],[2,7,0]],[[9]]]} & (\mathbb{Q}(\zeta_3)/\mathbb{Q},-2) \otimes (\mathbb{Q}(\zeta_4)/\mathbb{Q}, 1) \\
(48,17) & \mathtt{[1,Rationals,12,[[2,5,3],[2,7,6]],[[9]]]} & (\mathbb{Q}(\zeta_3)/\mathbb{Q},-2) \otimes (\mathbb{Q}(\zeta_4)/\mathbb{Q}, -1) \\
(48,18) & \mathtt{[1,Rationals,12,[[2,5,3],[2,7,6]],[[3]]]} & (\mathbb{Q}(\zeta_3)/\mathbb{Q},2) \otimes (\mathbb{Q}(\zeta_4)/\mathbb{Q}, -1) 
\end{array}$$

As one can see, all the tensor factors in the above example can directly converted into quaternion algebras, since $\mathbb{Q}(\zeta_3)=\mathbb{Q}(\sqrt{-3})$ and $\mathbb{Q}(\zeta_4)=\mathbb{Q}(\sqrt{-1})$.   Their local indices can be calculated directly with the Legendre symbol using {\tt LocalIndicesOfRationalQuaternionAlgebra}.  This makes it possible to directly verify the Schur index calculations done previously with characters and the dyadic Schur group classification.  

\medskip
To decompose the cyclotomic algebra {\tt [r,F,n,[[a$_1$,b$_1$,c$_1$],[a$_2$,b$_2$,c$_2$]],[[d]]]} when $\zeta_n^d \not\in \langle \zeta_4 \rangle$, the {\tt DecomposeCyclotomicAlgebra} function computes a scalar $c$ for adjusting $u$.  Let $m$ be the order of $\zeta_n^d$, and write $\zeta_n^d = \zeta_m^t$.  To find $c$ so that $v(cu) = (cu)v$, we need to find $c$ so that $c^{\sigma_{b_2}}\zeta_m^{tb_1b_2}=c$.  We try to construct a candidate for $c$ directly.  Start with the incomplete equation
$$ (1 + \dots) \zeta_m^{tb_1b_2} = (\zeta_m^{tb_1b_2} + \dots).  $$ 
The sum on the right-hand side represents $c$, and the sum on the left-hand side is $c^{\sigma_{b_2}}$.  We apply $\sigma_{b_2}^{-1} = \sigma_{b_2}^r$ to the term on the right-hand side, add that to the left-hand side, and produce a new term on the right-hand side: 
$$ (1 + \zeta_m^{tb_1} + \dots ) \zeta_m^{tb_1b_2} = (\zeta_m^{tb_1b_2} + \zeta_m^{tb_1(b_2+1)} + \dots). $$
Now continue.  Use the new term on the right to produce a new one on the left and another on the right: 
$$ (1 + \zeta_m^{tb_1} + \zeta_m^{tb_1(1+r)} + \dots ) \zeta_m^{tb_1b_2} = (\zeta_m^{tb_1b_2} + \zeta_m^{tb_1(b_2+1)} + \zeta_m^{tb_1(r+1+b_2)} + \dots). $$
And one more time to see the pattern: 
$$ (1 + \zeta_m^{tb_1} + \zeta_m^{tb_1(1+r)} +\zeta_m^{tb_1(r^2+r+1)} \dots ) \zeta_m^{tb_1b_2} = (\zeta_m^{tb_1b_2} + \zeta_m^{tb_1(b_2+1)} + \zeta_m^{tb_1(r^2+r+1+b_2)}\dots). $$
The process will eventually conclude when $\zeta_m^{tb_1(r^n+\dots+r+1)}=1$.  This produces a candidate for $c$ on the right, which will admit a tensor decomposition of the cyclotomic algebra into two cyclic algebra factors as long as $(cu)^{a_1} \in F^{\times}$. 
In the general algorithm for {\tt DecomposeCyclotomicAlgebra} in {\tt wedderga}, the latter condition is checked, and an error results if the condition fails.  Even when this function fails to find a scalar for $u$, the user can adjust the presentation of the cyclotomic algebra so that $v$ is listed before $u$, or so that $v$ has been adjusted by a scalar, and apply the function again to the new presentation.  

It may be desirable in some cases to be able to decompose cyclotomic algebras defined on galois groups with $3$ or more generators.  As of now there is no {\tt wedderga} feature for this.  For groups of order up to 511 this has not been necessary, as passing to the $p$-split subextension has reduced the galois groups to at most $2$ generators.

\end{document}